\documentclass[12pt]{article}
\usepackage{amssymb}
\usepackage{a4wide}
\usepackage{epsfig}
\usepackage{oldgerm}
\usepackage{amsmath}
\usepackage{cite}
\usepackage{amsthm}
\usepackage{fixmath}
\usepackage{multicol}
\usepackage{amstext,color}
\usepackage{array,multirow}
\usepackage{fullpage}
\usepackage{graphicx,graphics}
\usepackage{caption}
\usepackage{subcaption}

\usepackage[vlined,ruled,linesnumbered]{algorithm2e}

\definecolor{teal}{rgb}{0.0,0.5,0.5}
\definecolor{frenchrose}{rgb}{0.96,0.29,0.54}
\definecolor{lg}{rgb}{0.36,0.99,0.82}
\definecolor{dblue}{rgb}{0,0,.5}
\definecolor{dpink}{cmyk}{.2,1,.1,.04}
\definecolor{purple}{rgb}{0.35,0.04,0.64}
\definecolor{borange}{rgb}{1, .388, 0}
\definecolor{dpurple}{rgb}{0.61,0.22,1.00}
\definecolor{purp}{rgb}{0.44,0.00,0.87}
\definecolor{green}{rgb}{0.00,0.44,0.00}
\definecolor{junebud}{rgb}{0.74,0.85,0.34}
\definecolor{plum}{rgb}{0.56,0.27,0.52}
\theoremstyle{plain}
\newtheorem{theorem}{Theorem}[section]

\newtheorem{definition}[theorem]{Definition}
\newtheorem{lemma}[theorem]{Lemma}
\newtheorem{remark}[theorem]{Remark}

\newtheorem{example}[theorem]{Example}
\def\CC{{\textmd \kern.24em \vrule width.02em height1.4ex depth-.05ex \kern-.26emC}}
\def\TagOnRight
\def\QQ{\rlap {\raise 0.4ex \hbox{$\scriptscriptstyle |$}} {\hskip -0.1em Q}}
\catcode`\@=11
\begin{document}
%%%%%%%%%%%%%%%%%%%%%%%%%%%%%%%%
\begin{center}
  {{\bf \large {\rm {\bf $L2$-$1_{\sigma}$ Scheme on a Graded Mesh for a Multi-term Time-fractional Nonlocal Parabolic Problem }}}}
\end{center}
\begin{center}
%	{\textmd {{\bf Sudhakar Chaudhary,}}}\footnote{\it Department of Mathematics,  Institute of Infrastructure, Technology, Research and Management, Ahmedabad, Gujarat, India, (sudhakarchaudhary@iitram.ac.in)}
	{\textmd {\bf Pari J. Kundaliya}}\footnote{\it Department of Mathematics,  Institute of Infrastructure, Technology, Research and Management, Ahmedabad, Gujarat, India, (pariben.kundaliya.pm@iitram.ac.in)}
	
%	{\textmd {\bf Pari J. Kundaliya,}}\footnote{\it Department of Mathematics,  Institute of Infrastructure, Technology, Research and Management, Ahmedabad, Gujarat, India, (pariben.kundaliya.pm@iitram.ac.in)}
%	{\textmd {{\bf Sudhakar Chaudhary}}}\footnote{\it Department of Mathematics,  Institute of Infrastructure, Technology, Research and Management, Ahmedabad, Gujarat, India, (sudhakarchaudhary@iitram.ac.in)}
\end{center}
%%%%%%%%%%%%%%%%%%%%%%%%%%%%%%%%%%%%%%page 1 %%%%%%%%%%%%%%%%%%%%%%%%%%%%%%%%%%%%%%%%%%%%%%%%%
\begin{abstract}
In this article, we propose numerical scheme for solving a multi-term time-fractional nonlocal parabolic partial differential equation (PDE). The scheme comprises $L2$-$1_{\sigma}$ scheme on a graded mesh in time and Galerkin finite element method (FEM) in space. We present the discrete fractional Gr$\ddot{{o}}$nwall inequality for $L2$-$1_{\sigma}$ scheme in case of multi-term time-fractional derivative, which is a multi-term analogue of~\cite[Lemma 4.1]{[r16]}. We derive \textit{a priori} bound and error estimate for the fully-discrete solution. The theoretical results are confirmed via numerical experiments. We should note that, though the way of proving the discrete fractional Gr$\ddot{{o}}$nwall inequality is similar to~\cite{[r5]}, the calculation parts are more complicated in this article.
\end{abstract}
%\begin{center}

\noindent {\bf AMS(MOS):} 65M60, 65M15, 35R11.
%\end{center}
%%%%%%%%%%%%%%%%%%%%%%%%%%%%%%%%%%
\section{Introduction}\label{7intro}
%%%%% Needs corrections
Let $m \in \mathbb{N}$. Let $\mu_s > 0$ and $0~<~\alpha_{m}~<~\ldots~<~\alpha_1~<~1$, for~$s=1, \ldots, m$.  
 Consider the following time-fractional nonlocal diffusion equation:
 \begin{subequations}\label{7e1}
  	\begin{align}
  	\label{7cuc3:1.1}
  	\sum_{s=1}^{m}\, \mu_s \, ^{c}_{0}{D}^{\alpha_s}_{t} u(x,t) - a(l(u)) \: \Delta u(x,t) & =f(x,t) &  \mbox{in} & \quad \Omega\times(0,T], \\
  	\label{7cuc3:1.2}
   	u(x,t) & =0 &  \mbox{on} & \quad \partial\Omega\times(0,T],\\
   	\label{7cuc3:1.3}
  	u(x,0) & =u_0(x) &  \mbox{in} & \quad \Omega,
  	\end{align}
  \end{subequations}
  where $\Omega \subseteq \mathbb{R}^d$ $(d=1$ or $2)$ is a bounded domain with smooth boundary $\partial \Omega$, the function $a : \mathbb{R} \longrightarrow \mathbb{R}$ is a diffusion coefficient, and $l(u)= \int_{\Omega} u(x,t) \, dx.$ The term $^{c}_{0}{D}^{\alpha_s}_{t}u(x,t)$ is ${\alpha_s} ^{th}$ order Caputo fractional derivative of $u$, defined~\cite[Definition~3.1]{[r1]} as
  \begin{equation*}\label{7e2}
  \begin{split}
  {^{c}_{0}D^{\alpha_s}_{t}} u(x,t) =& \frac{1}{\Gamma (1-\alpha_s)}\int_{0}^{t}(t-\eta)^{-\alpha_s} \; \frac{\partial u(x,\eta)}{\partial \eta} \, d\eta , \;\; t>0.
  \end{split}
  \end{equation*}
  
  \vspace{15 pt}
  \noindent 
  Let us denote $^{c}_{0}{D}_{t} \, u := \sum_{s=1}^{m}\, \mu_s \, ^{c}_{0}{D}^{\alpha_s}_{t}u(x,t).$\\

If we choose the function $a(x) = 1$, then \eqref{7e1} reduces to the following multi-term time-fractional diffusion equation.
\begin{subequations}\label{7e4}
	\begin{align}
	\label{7cuc3:4.1}
	\sum_{s=1}^{m}\, \mu_s \, ^{c}_{0}{D}^{\alpha_s}_{t} u(x,t) - \Delta u(x,t) & =f(x,t) &  \mbox{in} & \quad \Omega\times(0,T], \\
	\label{7cuc3:4.2}
	u(x,t) & =0 &  \mbox{on} & \quad \partial\Omega\times(0,T],\\
	\label{7cuc3:4.3}
	u(x,0) & =u_0(x) &  \mbox{in} & \quad \Omega,
	\end{align}
\end{subequations}

It is claimed in~\cite{[Jin_2015]} that for the purpose of improving the modeling accuracy of the single-term model (resulting equation after choosing $m=1$ in \eqref{7e4}) for describing anomalous diffusion, \eqref{7e4} is developed. For example, using fractional dynamics for solute transport, a two-term fractional order diffusion model is proposed in~\cite{[Schumer_2003]} to distinguish explicitly mobile and immobile solutes. A kinetic equation involving two fractional derivatives of different orders also appears quite naturally when describing sub-diffusive motion in velocity fields~\cite{[Metzler_1998]}. 

In~\cite{[Jin_2015]}, \eqref{7e4} is solved using $L1$ scheme on a uniform mesh, which gives $O(\tau^{2-\alpha})$ convergence in time (where $\tau$ is the step size in temporal direction). In~\cite{[Alikhnov_2017]}, a second order accurate scheme in temporal direction namely $L2$-$1_{\sigma}$ scheme on a uniform mesh is developed to solve \eqref{7e4} numerically. The solution of time-fractional PDEs in general exhibits a weak singularity near time $t=0$~\cite{[kwn],[r02]}. This makes the existing numerical methods with a uniform time mesh often less accurate. The graded temporal meshes offer an efficient way of computing a reliable numerical solution near time $t=0$~\cite{[kwn]}. The $L1$ scheme on a graded mesh is used to solve \eqref{7e4} in~\cite{[Huang_2020]}~ and~\cite{[Huang_2022]}, which gives optimal convergence rate in case of weak singularity of a solution near $t=0$. The error analysis provided in~\cite{[Huang_2022]} is $\alpha_1$-robust; that is, the error bound remains valid if $\alpha_1$~$\rightarrow~{1}^{-}$~\cite{[r15]}. The $L2$-$1_{\sigma}$ scheme on a uniform mesh is used in~\cite{[z.Sun_2021]} to solve the multi-term time-fractional mixed diffusion and wave equation. In this paper, the fully-discrete $L2$-$1_{\sigma}$ scheme on a graded mesh is also given to solve the problem but its convergence analysis is not presented~\cite[Section 7]{[z.Sun_2021]}. 

In this work, we consider $L2$-$1_{\sigma}$ scheme on a graded mesh to find the numerical solution of \eqref{7e1}. To discretize a space variable, we use the Galerkin finite element method. In Section~\ref{7DFG}, we present the discrete fractional Gr$\ddot{{o}}$nwall inequality for the $L2$-$1_\sigma$ scheme on a graded mesh, which is an extension of~\cite[Lemma 4.1]{[r16]} to multi-term time-fractional derivative. We provide $\alpha_1$-robust \textit{a~priori} bound~(in Section~\ref{7bound}) and error estimate (in Section~\ref{7err_est}) for fully-discrete solution of \eqref{7e1}. 
%Moreover, here we are considering a nonlocal problem, to avoid application of the Newton's iteration at each step, we use linearization technique.     
	
Additionally, in this work, we assume that \eqref{7e1} has a unique solution with sufficient regularity (regularity assumption on $u$ is given in Section \ref{7err_est}), which is required for further analysis. We also need the following hypotheses on given data $a$, $f$ and $u_0$:
\begin{enumerate}
	\item[H1:] There exist $m_1, m_2 >0$ such that  $0<m_1 \le a(x) \le m_2 < \infty, \; \forall \, x \in \mathbb{R}.$
	\item[H2:] $a$ is Lipschitz continuous, i.e., $|a(x_1) - a(x_2)| \le L  |x_1-x_2|, \; (x_1, x_2 \in \mathbb{R})$.
	\item[H3:] $f \in L^{2}(0, T; L^2(\Omega))$ and $u_0 \in H_0^1(\Omega)\cap H^2(\Omega)$.
\end{enumerate}

%%%%
\section{Fully-discrete Scheme}\label{7FDS}
In this section, we discretize the differential equation \eqref{7e1} in both space and time. Firstly, we write the weak formulation of problem \eqref{7e1} and it is given below.\\

\noindent 
Find $u(\cdot,t) \in  H^1_0(\Omega)$ for each $t \in(0,T]$ such that 
\begin{subequations}\label{7e3}
\begin{align}
(^{c}_{0}{D}_{t} \, u, v ) \, + \, a\big(l(u)\big) \, (\nabla u, \nabla v) & =  \big(f, v\big),&   &\forall v \in H^1_0(\Omega),&  & & \\ 
\big(u(\cdot, 0), v\big) & =  \big(u_0 (\cdot), v\big), &  &\mbox{in} \; \, \Omega.&   & & 
\end{align}
\end{subequations}

In order to establish a fully-discrete scheme, we assume that $\tau_h$ is a quasi uniform partition of $\Omega$ into disjoint intervals in $\mathbb{R}^1$ or triangles in $\mathbb{R}^2$ and $h$ be the mesh size. We define $V_h$ as a finite dimensional subspace of $H^{1}_{0}(\Omega)$, with dimension $M$, consisting of continuous piecewise linear polynomials.

Let $\tau_{t_n} := \Big\{ t_n : t_n = T\Big(\frac{n}{N}\Big)^{r}, \, n=0, 1,..., N\Big\}$ be non-uniform partition of $[0,T]$ with time step $\tau_n = t_n - t_{n-1}$, for $n=1,2,...,N$ and mesh grading parameter $r \ge 1$. The step size $\tau_n$ and grid point $t_n$ satisfy the following estimates~\cite[Eq (18) and (19)]{[fmr13]}:
\begin{equation}\label{7e20}
\begin{split}
\tau_n \le C \, r \, T \,  N^{-r} n^{r-1} \, \le \, C \, r \, T^{\frac{1}{r}} \, N^{-1} \, t_n^{1 - \frac{1}{r}}, \quad \mbox{for} \; 1 \le n \le N, \\
\end{split}
\end{equation}

and 
\begin{equation}\label{7e20a}
\begin{split}
t_n \le 2^r \, t_{n-1}, \quad \mbox{for} \; 2 \le n \le N. \\
\end{split}
\end{equation}

From the estimate \eqref{7e20a}, it follows that there is a constant $C_r > 0,$ independent of the step size $\tau_n$ such that
\begin{equation}\label{7e20b}
\begin{split}
\tau_{n-1} \le \tau_n  \le C_r \, \tau_{n-1}, \quad \mbox{for} \; 2 \le n \le N. \\
\end{split}
\end{equation}

Now, for each $n=1, \ldots, N$, we define the function
\begin{eqnarray}\label{7e125}
   \mathcal{G}_n(\sigma) = \sum_{s=1}^{m} \frac{\mu_s}{\Gamma({3-\alpha_s})} \sigma^{1-\alpha_s} \big[\sigma - \big(1-\frac{\alpha_s}{2}\big)\big] \tau_{n}^{2-\alpha_s}, \qquad \sigma \ge 0, \nonumber
\end{eqnarray}
and take
\begin{eqnarray}
 b_1 & = & \min_{1 \le s \le m} \big\{ 1- \frac{\alpha_s}{2}\big\}  \; = \; 1 - \frac{\alpha_1}{2} \; \ge \;  \frac{1}{2}, \nonumber \\
 b_2 & = & \max_{1 \le s \le m} \big\{ 1- \frac{\alpha_s}{2}\big\}  \; =  \; 1 - \frac{\alpha_m}{2}  \; \le \;  1. \nonumber
\end{eqnarray}
As given in~\cite[Lemma 2.1]{[Alikhnov_2017]}, $\mathcal{G}_n(\sigma)=0$ has a unique root (say $\sigma_n^{\star}$) in $[b_1, b_2]$ depending upon $n$. Using Newton's method and by taking $b_2$ as an initial guess, one can find $\sigma_n^{\star}$.  

Let $\sigma_n = 1 - \sigma_n^{\star}$. For any function $v \in C[0,T] \cap C^3(0,T],$ 
we write $v(t_n):=v^n.$ For $n=1, \ldots, N-1$, let $L_{2,n} v(\eta)$ be the quadratic polynomial that interpolates to $v(\eta)$ at the points $t_{n-1}, t_n$ and $t_{n+1}$. Also, for simplicity of writing, we introduce the following notations. For each $n = 1, \ldots, N$
\begin{eqnarray}
%\text{For each $n = 1, \ldots, N$} \nonumber \\
	 t_{n-\sigma_n} & := & \, t_{n-1} \, + \, (1 - \sigma_n) \, \tau_n \, = \, (1 - \sigma_n) \, t_n \, + \,  \sigma_n \, t_{n-1} \nonumber \\
	 v^{n, \sigma_n} & := & \, (1 - \sigma_n) \, v^n \, + \,  \sigma_n \, v^{n-1} \nonumber \\
	 \delta_t v^n & := & \, \frac{v^n - v^{n-1}}{\tau_n} \nonumber \\
 \hspace{-38pt} \text{and for each $n = 2, \ldots, N$} \nonumber \\ 
	 \widehat{v}^{n - \sigma_n} & := & \frac{\tau_{n-1} \, + \, (1 - \sigma_n) \, \tau_n}{\tau_{n-1}} \, v^{n-1} \, - \, \frac{(1 - \sigma_n) \, \tau_n}{\tau_{n-1}} \, v^{n-2} \nonumber 
\end{eqnarray}

Now, for each $s=1,2,...,m,$ the $L2$-$1_{\sigma}$ approximation to $^c_0D^{\alpha_s}_{ t_{n-\sigma_n}}v$ on the graded mesh~\cite[P.5]{[r16]} is given below. 

For $n = 1, \ldots, N,$
\begin{eqnarray}\label{7e39}
^c_0D^{\alpha_s}_{ t_{n-\sigma_n}}v &=&  \frac{1}{\Gamma{(1-\alpha_s)}} \, \int_{0}^{t_{n-\sigma_n}} \, (t_{n-\sigma_n} - \eta)^{- \alpha_s} \, v'(\eta) \, d\eta \nonumber \\
&=&  \frac{1}{\Gamma{(1-\alpha_s)}} \, \sum_{j=1}^{n-1} \int_{t_{j-1}}^{t_j} \,  (t_{j-\sigma_n} - \eta)^{- \alpha_s} \, v'(\eta) \, d\eta \nonumber \\
& & + \, \frac{1}{\Gamma{(1-\alpha_s)}} \, \int_{t_{n-1}}^{t_{n-\sigma_n}} \, (t_{n-\sigma_n} - \eta)^{- \alpha_s} \, \delta_t v^n \, d\eta \nonumber \\
&=& \sum_{j=1}^{n} g^{(\alpha_s, \sigma_n)}_{n,j} (v^j - v^{j-1})\nonumber \\
&=& \: g^{(\alpha_s, \sigma_n)}_{n,n} \, v^n - g^{(\alpha_s, \sigma_n)}_{n,1} \, v^0 - \sum_{j=2}^{n} \, \big( g^{(\alpha_s, \sigma_n)}_{n,j} - g^{(\alpha_s, \sigma_n)}_{n,j-1}\big) \, v^{j-1} \nonumber \\
&:=& D^{\alpha_s}_{N} v^{n- \sigma_n}.
\end{eqnarray}

Here $g^{(\alpha_s, \sigma_n)}_{1,1} = \tau_1^{-1} \, a^{(\alpha_s, \sigma_n)}_{1,1}$ and for $n \ge 2,$
\begin{eqnarray}\label{7e40}
g^{(\alpha_s, \sigma_n)}_{n,j}= \left\{\begin{array}{lll} 
\tau_j^{-1} \, \big(a^{(\alpha_s, \sigma_n)}_{n,1}-b^{(\alpha_s, \sigma_n)}_{n,1}\big) & (\text{if } j=1)  \vspace{5pt} \\ 
\tau_j^{-1} \, \big(a^{(\alpha_s, \sigma_n)}_{n,j}+b^{(\alpha_s, \sigma_n)}_{n,j-1}-b^{(\alpha_s, \sigma_n)}_{n,j}\big) & (\text{if } \; 2 \le j \le n-1) \vspace{5pt} \\  
\tau_j^{-1} \, \big(a^{(\alpha_s, \sigma_n)}_{n,n}+b^{(\alpha_s, \sigma_n)}_{n,n-1}\big) & (\text{if } j=n),
\end{array}
\right.
\end{eqnarray}
where
\begin{eqnarray}
a^{(\alpha_s, \sigma_n)}_{n,n} \! & = & \! \frac{1}{\Gamma{(1-\alpha_s)}} \, \int_{t_{n-1}}^{t_{n-\sigma_n}} \, (t_{n-\sigma_n} - \eta)^{- \alpha_s} \, d\eta, \nonumber \\
\text{and for $1 \le j \le n-1$,} \nonumber \\
a^{(\alpha_s, \sigma_n)}_{n,j} \! & = & \! \frac{1}{\Gamma{(1-\alpha_s)}} \, \int_{t_{j-1}}^{t_j} \, (t_{n-\sigma_n} - \eta)^{- \alpha_s} \, d\eta, \nonumber \\
b^{(\alpha_s, \sigma_n)}_{n,j} \! & = & \! \frac{2}{\Gamma{(1-\alpha_s)} (t_{j+1} - t_{j-1})} \int_{t_{j-1}}^{t_j} (t_{n-\sigma_n} - \eta)^{- \alpha_s} (\eta - t_{j-\frac{1}{2}}) \, d\eta. \nonumber 
\end{eqnarray}
%\begin{eqnarray}
%a^{(\alpha_s, \sigma_n)}_{n,n} \! & = & \! \frac{1}{\Gamma{(1-\alpha_s)}} \, \int_{t_{n-1}}^{t_{n-\sigma_n}} \, (t_{n-\sigma_n} - \eta)^{- \alpha_s} \, d\eta, \; \; \mbox{for} \; n \ge 1. \nonumber \\
%a^{(\alpha_s, \sigma_n)}_{n,j} \! & = & \! \frac{1}{\Gamma{(1-\alpha_s)}} \, \int_{t_{j-1}}^{t_j} \, (t_{n-\sigma_n} - \eta)^{- \alpha_s} \, d\eta, \; \; \mbox{for}  \; 1 \le j \le n-1. \nonumber \\
%b^{(\alpha_s, \sigma_n)}_{n,j} \! & = & \! \frac{2}{\Gamma{(1-\alpha_s)} (t_{j+1} - t_{j-1})} \int_{t_{j-1}}^{t_j} (t_{n-\sigma_n} - \eta)^{- \alpha_s} (\eta - t_{j-\frac{1}{2}}) \, d\eta, \, \mbox{ for }  \, 1 \le j \le n-1. \nonumber 
%\end{eqnarray}

\vspace{20pt}
 \indent 
Additionally, the following two results hold for $g^{(\alpha_s, \sigma_n)}_{n,j}$~\cite[Example 3, P.222]{[r5]}.
\begin{eqnarray}\label{7e7a}
  g^{(\alpha_s, \sigma_n)}_{n,n} \, \ge \, g^{(\alpha_s, \sigma_n)}_{n,n-1} \, \ge  \ldots \ge \, g^{(\alpha_s, \sigma_n)}_{n,2} \, \ge \, g^{(\alpha_s, \sigma_n)}_{n,1} \, > \, 0, \quad \mbox{for} \: 1 \le n \le N,
\end{eqnarray}
and if $\max\limits_{2 \le n \le N} \frac{\tau_{n-1}}{\tau_n} \, \le \frac{7}{4}$, then
\begin{eqnarray}\label{7e7b}
g^{(\alpha_s, \sigma_n)}_{n,j} \, \ge \, \frac{4}{11 \, \tau_j} \, \int_{t_{j-1}}^{t_j} \frac{(t_n - \eta)^{- \alpha_s}}{\Gamma (1 - \alpha_s)} \, d\eta, \quad \mbox{for} \: 1 \le j \le n \le N.
\end{eqnarray}

Now, for each $n=1, \ldots, N$ and $j=1, \ldots, n,$ if we denote
\begin{eqnarray}\label{7e126}
  g_{n,j} := \sum_{s=1}^{m} \mu_s \, g^{(\alpha_s, \sigma_n)}_{n,j}, 
\end{eqnarray}
then 
\begin{eqnarray}\label{7e127}
^{c}_{0}{D}_{t_{n- \sigma_n}}v & \approx & \sum_{j=1}^{n} g_{n,j} (v^j - v^{j-1}) \nonumber \\
& = &  g_{n,n} \, v^n - g_{n,1} \, v^{0} - \sum_{j=2}^{n} \, \big( g_{n,j} - g_{n,j-1}\big) \, v^{j-1} \nonumber \\
& := & D_{N} \, v^{n- \sigma_n}. 
\end{eqnarray}

Moreover, as a result of \eqref{7e7a} and \eqref{7e7b}, we have the following two properties for $g_{n,j}$.
\begin{enumerate}
  \item $g_{n,j}$ are positive and monotone.
        \begin{eqnarray}\label{7e8a}
           g_{n,n} \, \ge \, g_{n,n-1} \, \ge  \ldots \ge \, g_{n,2} \, \ge \, g_{n,1} \, > \, 0, \quad \mbox{for} \: 1 \le n \le N,
        \end{eqnarray}
  \item If $\max\limits_{2 \le n \le N} \frac{\tau_{n-1}}{\tau_n} \, \le \frac{7}{4}$ and let $\mu = \min\limits_{1 \le s \le m} \frac{\mu_s}{\Gamma(2 - \alpha_s)}$, then for $1 \le j \le n \le N$,
  \begin{eqnarray}\label{7e8b}
  g_{n,j} \; \ge \; \frac{4}{11 \, \tau_j} \, \sum_{s=1}^{m} \, \mu_s \int_{t_{j-1}}^{t_j} \frac{(t_n - \eta)^{- \alpha_s}}{\Gamma (1 - \alpha_s)} \, d\eta  \;  \ge \; \frac{4 \, \mu}{11} \, \tau_j^{- \alpha_1}. 
  \end{eqnarray}
\end{enumerate}

To find the discretization error $^{c}_{0}{D}_{t_{n- \sigma_n}}v - D_{N} \, v^{n- \sigma_n}$, for each $\alpha_s \, (s=1,2,...,m)$ we define
\begin{subequations}\label{7e128}
 	\begin{align}
 	\label{7e128:1.1}
 	   \psi^{1 - \sigma_1}_{\alpha_s} & = \tau_1^{\alpha_s} \sup_{\eta \in (0,t_1)} \big(\eta^{1-\alpha_s} |\delta_t v(t_1) - v'(\eta)| \big),&  & \mbox{if} &   & n=1, &  \\
 	\label{7e128:1.2}
       \psi^{n-\sigma_n}_{\alpha_s} & =  \tau_n^{3-\alpha_s} \, t_{n-\sigma_n}^{\alpha_s} \sup_{\eta \in (t_{n-1},t_n)} |v^{(3)}(\eta)|,&  & \mbox{if} &  &n=2, \ldots, N, &  \\
 	\label{7e128:1.3}
 	    \psi^{n,1}_{\alpha_s} & =  \tau_1^{\alpha_s} \sup_{\eta \in (0,t_1)} \big(\eta^{1-\alpha_s} | (L_{2,1} v(\eta))' - v'(\eta)| \big), & & \mbox{if}&  & n=2, \ldots, N, & \\
 	\label{7e128:1.4}
 	    \psi^{n,j}_{\alpha_s} & =  \tau_n^{-\alpha_s} \, \tau_j^2 (\tau_j + \tau_{j+1}) \, t_j^{\alpha_s} \! \sup_{\eta \in (t_{j-1},t_j)} |v^{(3)}(\eta)|, & & \mbox{if} & & 2 \le j < n \le N. & 
 	\end{align}
 \end{subequations}

\begin{lemma}\label{7l6a}
	\cite[Lemma 7]{[fmr13]} Let $v(t) \in C[0,T] \cap C^3(0,T]$ be a function with $|v^{(q)}(t)| \le C (1 + t^{\alpha_1 - q}),$ for $q=0,1,2,3$ and for $t \in (0,T],$ then 
	\begin{eqnarray}\label{7e129}
	   \psi^{n-\sigma_n}_{\alpha_s}  & \le &  C N^{-\min\{3-\alpha_s, r \alpha_1\}}, \quad  \mbox{for} \; \, n=1, \ldots, N, \nonumber \\
	   \psi^{n,j}_{\alpha_s} & \le &  C N^{-\min\{3-\alpha_s, r \alpha_1\}}, \quad \mbox{for} \; \,  j=1, \ldots, n-1, \; \mbox{when} \; n \ge 2. \nonumber
	\end{eqnarray}		
\end{lemma}

Now, in the following theorem, we derive the estimate for $\|^{c}_{0}{D}_{t_{n- \sigma_n}}v - D_{N} \, v^{n- \sigma_n}\|$. The proof is slight modification of the proof~\cite[Lemma 1]{[fmr13]}.
\begin{theorem}\label{7l6}
	 Let $v(t) \in C[0,T] \cap C^3(0,T]$ be a function with $|v^{(q)}(t)| \le C (1 + t^{\alpha_1 - q}),$ for $q=0,1,2,3$ and for $t \in (0,T],$ then for each $n= 1, \ldots, N$,
	\begin{eqnarray}\label{7e130}
	 \| ^c_0D_{ t_{n-\sigma_n}}v - D_{N} \, v^{n- \sigma_n} \| \, \le C \, N^{-\min \{3- \alpha_1, \, r \alpha_1\}} \, \sum_{s=1}^{m} \frac{\mu_s}{\Gamma{(1-\alpha_s)}} \, t^{-\alpha_s}_{n - \sigma_n}.
	\end{eqnarray}		
\end{theorem}

\begin{proof} 
Let $e^{n- \sigma_n}:= \, ^c_0D_{ t_{n- \sigma_n}}v - D_{N} \, v^{n- \sigma_n}.$ Then we have
\begin{eqnarray}\label{7e131}
  e^{n- \sigma_n} & = & \sum_{s=1}^{m} \frac{\mu_s}{\Gamma(1-\alpha_s)} \sum_{j=1}^{n-1} \int_{t_{j-1}}^{t_j} (t_{n- \sigma_n} - \eta)^{-\alpha_s} \big(L_{2,j} v(\eta) - v(\eta)\big)' d\eta \nonumber \\
  & & + \, \sum_{s=1}^{m} \frac{\mu_s}{\Gamma(1-\alpha_s)} \int_{t_{n-1}}^{t_{n- \sigma_n}} (t_{n- \sigma_n} - \eta)^{-\alpha_s} \big(\delta_t v(t_n) - v'(\eta)\big) d\eta.
\end{eqnarray}

 For $n=1$,
\begin{eqnarray}\label{7e132}
 | e^{1 - \sigma_1} |& = & \, \Big| \sum_{s=1}^{m} \frac{\mu_s}{\Gamma(1-\alpha_s)} \int_{t_0}^{t_{1 - \sigma_1}} (t_{1 - \sigma_1} - \eta)^{-\alpha_s} \big(\delta_t v(t_1) - v'(\eta)\big) d\eta \Big| \nonumber \\
 & \le & \sum_{s=1}^{m} \frac{\mu_s}{\Gamma(1-\alpha_s)} \sup_{\eta \in (0,t_1)} \! \! \big(\eta^{1-\alpha_s} |\delta_t v(t_1) - v'(\eta)| \big) \! \! \int_{t_0}^{t_{1 - \sigma_1}} \! (t_{1 - \sigma_1} - \eta)^{-\alpha_s} \, \eta^{\alpha_s - 1} d\eta \nonumber \\
 & \le & \sum_{s=1}^{m} \frac{\mu_s}{\Gamma(1-\alpha_s)} \, \psi^{1 - \sigma_1}_{\alpha_s} \, t_{1 - \sigma_1}^{-\alpha_s} \, \Gamma(1-\alpha_s) \, \Gamma(\alpha_s). 
\end{eqnarray}

When $2 \le n \le N$, \\

\noindent For each $s=1,2,...,m$, let us write
\begin{eqnarray}\label{7e133}
\mathcal{A}_s = \sum_{j=1}^{n-1} \int_{t_{j-1}}^{t_j} (t_{n- \sigma_n} - \eta)^{-\alpha_s} \big(L_{2,j} v(\eta) - v(\eta)\big)' d\eta.
\end{eqnarray}

\noindent Then by following steps given in the proof of~\cite[Lemma 1]{[fmr13]} and using the interpolation error estimate for the term $\big(L_{2,j} v(\eta) - v(\eta)\big)$, one can show that
\begin{eqnarray}\label{7e134}
| \mathcal{A}_s | \le  t_{n- \sigma_n}^{-\alpha_s} \, \max_{1 \le j \le n-1} \psi^{n,j}_{\alpha_s}.
\end{eqnarray}

\noindent Now, we denote 
\begin{eqnarray}\label{7e135}
\mathcal{B} = \sum_{s=1}^{m} \frac{\mu_s}{\Gamma(1-\alpha_s)} \int_{t_{n-1}}^{t_{n- \sigma_n}} (t_{n- \sigma_n} - \eta)^{-\alpha_s} \big(\delta_t v(t_n) - v'(\eta)\big) d\eta.
\end{eqnarray}

\noindent Therefore, using the Taylor's theorem we can have
\begin{eqnarray}\label{7e136}
\mathcal{B} & = & \sum_{s=1}^{m} \frac{\mu_s}{\Gamma(1-\alpha_s)} \Bigg[ \int_{t_{n-1}}^{t_{n- \sigma_n}} (t_{n- \sigma_n} - \eta)^{-\alpha_s} \big(\delta_t v(t_n) - v'(t_{n-\frac{1}{2}})\big) d\eta \nonumber \\
& & - \: v''(t_{n-\frac{1}{2}}) \, \int_{t_{n-1}}^{t_{n- \sigma_n}} (t_{n- \sigma_n} - \eta)^{-\alpha_s} (\eta - t_{n-\frac{1}{2}}) d\eta \nonumber \\
& & - \: \frac{1}{2} \int_{t_{n-1}}^{t_{n- \sigma_n}} (t_{n- \sigma_n} - \eta)^{-\alpha_s} (\eta - t_{n-\frac{1}{2}})^2 \, v'''(\zeta_n(\eta)) d\eta \Bigg],
\end{eqnarray}
where $\zeta_n(\eta) \in (t_{n-1}, t_{n- \sigma_n}).$ \\

\noindent An application of Taylor's theorem into the first term of R.H.S. of \eqref{7e136} gives
\begin{eqnarray}\label{7e137}
& & \Big| \sum_{s=1}^{m} \frac{\mu_s}{\Gamma(1-\alpha_s)}  \int_{t_{n-1}}^{t_{n- \sigma_n}} (t_{n- \sigma_n} - \eta)^{-\alpha_s} \big(\delta_t v(t_n) - v'(t_{n-\frac{1}{2}})\big) d\eta \Big| \nonumber \\
& & \; \le \, \sum_{s=1}^{m} \frac{\mu_s}{\Gamma(1-\alpha_s)} \, \tau_{n}^2 \, \sup_{\eta \in (t_{n-1},t_n)} |v'''(\eta)| \, \int_{t_{n-1}}^{t_{n- \sigma_n}} (t_{n- \sigma_n} - \eta)^{-\alpha_s} d\eta \nonumber \\
& & \; \le \,  \sum_{s=1}^{m} \frac{\mu_s}{\Gamma(1-\alpha_s)} \, \tau_n^{3-\alpha_s} \, \sup_{\eta \in (t_{n-1},t_n)} |v'''(\eta)| \nonumber \\
& & \; = \, \sum_{s=1}^{m} \frac{\mu_s}{\Gamma(1-\alpha_s)} \, t_{n- \sigma_n}^{-\alpha_s} \, \psi^{n- \sigma_n}_{\alpha_s}.
\end{eqnarray}

\noindent In order to get the estimate for the middle term of \eqref{7e136}, we need to evaluate the following integral using integration by parts.
\begin{eqnarray}\label{7e138}
 & & \sum_{s=1}^{m} \frac{\mu_s}{\Gamma(1-\alpha_s)} \, \int_{t_{n-1}}^{t_{n- \sigma_n}} (t_{n- \sigma_n} - \eta)^{-\alpha_s} (\eta - t_{n-\frac{1}{2}}) d\eta \nonumber \\ 
 & & \; = \,  \sum_{s=1}^{m} \frac{\mu_s}{\Gamma(1-\alpha_s)} \, \Bigg[ \frac{(t_{n- \sigma_n} - t_{n-1})^{1-\alpha_s} \, (t_{n-1} - t_{n-\frac{1}{2}})}{1-\alpha_s} + \frac{(t_{n- \sigma_n} - t_{n-1})^{2-\alpha_s}}{(1-\alpha_s) \, (2-\alpha_s)} \Bigg] \nonumber \\
 & & \; = \, \sum_{s=1}^{m} \frac{\mu_s}{\Gamma(1-\alpha_s)} \, \Bigg[ \frac{(\alpha_s -2) \, (1 - \sigma_n)^{1-\alpha_s} \, \tau_n^{2-\alpha_s}}{2 (1-\alpha_s) (2-\alpha_s)} + \frac{ (1 - \sigma_n)^{2-\alpha_s} \tau_n^{2-\alpha_s}}{ (1-\alpha_s) (2-\alpha_s)} \Bigg] \nonumber  \\
 & & \; = \, \sum_{s=1}^{m} \frac{\mu_s}{\Gamma(1-\alpha_s)} \, (\sigma_n^{\star})^{1-\alpha_s} \, \tau_n^{2-\alpha_s} \, \Big( \sigma_n^{\star} - \big(1- \frac{\alpha_s}{2} \big) \Big) \nonumber  \\
 & & \;= \, 0, \quad \mbox{as $\sigma_n^{\star}$ is a root of the function $\mathcal{G}_n.$ }
\end{eqnarray}

\noindent Also, for the third term of \eqref{7e136}, we have 
\begin{eqnarray}\label{7e139}
  & & \Big| \sum_{s=1}^{m} \frac{\mu_s}{\Gamma(1-\alpha_s)} \, \frac{1}{2} \int_{t_{n-1}}^{t_{n- \sigma_n}} (t_{n- \sigma_n} - \eta)^{-\alpha_s} (\eta - t_{n-\frac{1}{2}})^2 \, v'''(\zeta_n(\eta)) d\eta \Big| \nonumber \\
  & & \; \le \, \sum_{s=1}^{m} \frac{\mu_s}{\Gamma(1-\alpha_s)} \, \tau_n^2 \, \sup_{\eta \in (t_{n-1},t_n)} |v'''(\eta)| \, \int_{t_{n-1}}^{t_{n- \sigma_n}} (t_{n- \sigma_n} - \eta)^{-\alpha_s} d\eta \nonumber  \\
  & & \; = \, \sum_{s=1}^{m} \frac{\mu_s}{\Gamma(1-\alpha_s)} \, t_{n- \sigma_n}^{-\alpha_s} \, \psi^{n- \sigma_n}_{\alpha_s}.
\end{eqnarray}

\noindent Using the results obtaining in \eqref{7e137}-\eqref{7e139} into \eqref{7e136}, we can arrive at 
\begin{eqnarray}\label{7e140}
  |\mathcal{B}| \, \le \, \sum_{s=1}^{m} \frac{\mu_s}{\Gamma(1-\alpha_s)} \, t_{n- \sigma_n}^{-\alpha_s} \, \psi^{n- \sigma_n}_{\alpha_s}. 
\end{eqnarray}

\noindent Combining the inequalities \eqref{7e132}, \eqref{7e134} and \eqref{7e140}, it follows that  
\begin{eqnarray}\label{7e141}
|e^{n- \sigma_n}| \, \le \, \sum_{s=1}^{m} \frac{\mu_s}{\Gamma(1-\alpha_s)} \, t_{n- \sigma_n}^{-\alpha_s} \, \big( \psi^{n- \sigma_n}_{\alpha_s} + \max_{1 \le j \le n-1} \psi^{n,j}_{\alpha_s} \big).
\end{eqnarray}

\noindent An application of Lemma \ref{7l6a} into \eqref{7e141} leads to 
\begin{eqnarray}\label{7e142}
|e^{n- \sigma_n}| & \le & C \, \sum_{s=1}^{m} \frac{\mu_s}{\Gamma(1-\alpha_s)} \, t_{n- \sigma_n}^{-\alpha_s} \, N^{-\min\{3-\alpha_s, r \alpha_1\}} \nonumber \\
& \le & C N^{-\min\{3-\alpha_1, r \alpha_1\}} \, \sum_{s=1}^{m} \frac{\mu_s}{\Gamma(1-\alpha_s)} \, t_{n- \sigma_n}^{-\alpha_s}. 
\end{eqnarray}
This completes the proof.
\end{proof}

%%%%%% Fully-disccrete scheme
Using the above notations and let $U^n$ (for $n= 0, \ldots, N$) denotes the approximate value of $u$ at time $t_n$, now we write the \textbf{fully-discrete scheme} for \eqref{7e1} as follows. For each $n=1, \ldots, N$, find $U^n \in V_h$ such that $\forall v_h \in V_h$,
\begin{subequations}\label{7e6}
\begin{align}
 \label{7e6:1.1}
     U^0 & = u^0_h,& \\
 \label{7e6:1.2}
     \left( {D}_{N} \, U^{1- \sigma_1}, v_h \right) \, + \, a\big(l(U^{1, \sigma_1})\big) \, (\nabla U^{1, \sigma_1}, \nabla v_h) & =  (f^{1- \sigma_1}, v_h),& \\
   \label{7e6:1.3}
	\left( {D}_{N} \, U^{n- \sigma_n}, v_h \right) \, + \, a\big(l(\widehat{U}^{n - \sigma_n})\big) \, (\nabla U^{n, \sigma_n}, \nabla v_h) & = (f^{n- \sigma_n}, v_h),&  \mbox{for} \; n \ge 2.
\end{align}
\end{subequations}
where $u^0_h$ is some approximation of $u_0(x)$. 

As mentioned in \cite{[r3],[sk],[me],[sp1]}, a direct application of Newton's method to solve a nonlocal problem is computationally expensive. Due to this reason, first we need to reformulate the original problem. To avoid the reformulation and Newton's iteration at each time level, here we use the idea introduced in \cite{[r3]} to find $U^1$ only. Once $U^1$ is obtained, in order to get $U^n$ (for $n \ge 2$), we use a linearization technique. 

%%%%%
\section{A Discrete Fractional Gr$\ddot{{o}}$nwall Inequality}\label{7DFG}
In this section, we present the discrete fractional Gr$\ddot{{o}}$nwall inequality, which plays an important role in the derivation of \textit{a priori} bound and convergence result. 

For this, we define the coefficients $p_{n,i}$ (for $n=1,2,...,N$) such that
\begin{eqnarray}\label{7e121a}
   \sum_{k=j}^{n} \, p_{n,k} \, g_{k,j} \equiv 1, \quad \mbox{for} \: 1 \le j \le n \le N.
\end{eqnarray}

Therefore, taking $j=i$ and $j=i+1$ in \eqref{7e121a}, we can define $p_{n,i}$ as follows:
\begin{eqnarray}\label{7e121}
p_{n,i} :=  \left\{\begin{array}{ll} \frac{1}{g_{i,i}} \sum_{k=i+1}^{n} \left( g_{k, i+1} - g_{k, i} \right) p_{n,k} & (\mbox{if} \; i=1, 2,...,n-1), \vspace{5pt} \\
\frac{1}{g_{n,n}} & (\mbox{if} \; \, i=n).
\end{array}\right.
\end{eqnarray}

Moreover, as $p_{n,i} \, g_{i,i}$ $\le$ $\sum_{k=i}^{n} \, p_{n,k} \, g_{k,i}$ $=$ $1$, 
\begin{eqnarray}\label{7e122}
0 \, < \, p_{n,i} \, \le \, \frac{1}{g_{i,i}} \, \le \frac{11}{4 \, \mu} \, \tau_i^{\alpha_1}.
\end{eqnarray}

In the following two Lemmas, we state and prove the properties of $p_{n,k}$.
 These lemmas are generalization of~\cite[Lemmas 2.1 and 2.2]{[r5]} to multi-term time-fractional derivative. The proofs are similar to~\cite[Lemmas 2.1 and 2.2]{[r5]} with some modifications.
 
%\textcolor{red}{In the following two Lemmas, we state and prove the properties of $p_{n,k}$.
%	One should keep in mind that the results given in~\cite[Lemmas 2.1 and 2.2]{[r5]} are for single-term time-fractional derivative and can't be directly use here. For further analysis, we need to extend these results for $p_{n,k}$, in case of multi-term time-fractional derivative. The proof follows in the similar steps to~\cite[Lemmas 2.1 and 2.2]{[r5]}, as in case of single-term time-fractional derivative. However, some modifications are also required.}

\begin{lemma}\label{7l2}
  Let $v: [0,T] \longrightarrow \mathbb{R}$ be a continuous, piecewise-$C^1$ function.
  \begin{enumerate}
   \item[(1)] If $v'$ is nonnegative and decreasing, then for $1 \le n \le N$, 
     \begin{eqnarray}\label{7e5}
       \sum_{j=1}^{n} \, p_{n,j} \, ^{c}_{0}{D}_{t_j} v \, \le \, \frac{11}{4} \, \int_{0}^{t_n} v'(\eta) \, d\eta.
     \end{eqnarray}
    \item[(2)] If $v'$ is nonnegative and monotonic, then for $1 \le n \le N$,
      \begin{eqnarray}\label{7e5a}
      \sum_{j=1}^{n-1} \, p_{n,j} \, ^{c}_{0}{D}_{t_j} v \, \le \, \frac{11}{4} \, \int_{0}^{t_n} v'(\eta) \, d\eta.
      \end{eqnarray}
  \end{enumerate}
\end{lemma}
\begin{proof}
   \textit{(1)} Assume that $v'$ is nonnegative and decreasing.
   \begin{eqnarray}\label{7e7}
     ^{c}_{0}{D}_{t_j} v & = & \sum_{s=1}^{m} \, \mu_s \, ^{c}_{0}{D}^{\alpha_s}_{t_j} v \nonumber \\
     & = & \sum_{s=1}^{m} \, \mu_s \, \sum_{k=1}^{j} \int_{t_{k-1}}^{t_k} \frac{(t_j - \eta)^{- \alpha_s}}{\Gamma{(1- \alpha_s)}} \, v'(\eta) \, d\eta.
  \end{eqnarray}
  
  \noindent Since $\frac{(t_j - \eta)^{- \alpha_s}}{\Gamma{(1- \alpha_s)}}$ is increasing and $v'$ is decreasing function, using Chebyshev's sorting inequality~\cite[P. 168, Item 236]{[inequalities]} in \eqref{7e7}, we can arrive at
  	\begin{eqnarray}\label{7e10}
  	  ^{c}_{0}{D}_{t_j} v & \le & \sum_{s=1}^{m} \, \mu_s \, \Big\{ \sum_{k=1}^{j} \, \Big( \frac{1}{\tau_k} \int_{t_{k-1}}^{t_k} \frac{(t_j - \eta_1)^{- \alpha_s}}{\Gamma{(1- \alpha_s)}} \, d\eta_1 \Big) \, \Big( \int_{t_{k-1}}^{t_k} v'(\eta) \, d\eta \Big) \Big\}.
  	\end{eqnarray}
  	
  	\noindent After rearranging the terms of \eqref{7e10}, we can obtain
  	\begin{eqnarray}\label{7e11}
  	^{c}_{0}{D}_{t_j} v & \le & \sum_{k=1}^{j} \, \Big\{ \Big( \sum_{s=1}^{m} \, \frac{\mu_s}{\tau_k} \int_{t_{k-1}}^{t_k} \frac{(t_j - \eta_1)^{- \alpha_s}}{\Gamma{(1- \alpha_s)}} d\eta_1 \Big) \, \Big( \int_{t_{k-1}}^{t_k} v'(\eta) \, d\eta \Big) \Big\} \nonumber \\
  	&\le & \frac{11}{4} \, \sum_{k=1}^{j} \, g_{j,k} \, \Big( \int_{t_{k-1}}^{t_k} v'(\eta) \, d\eta \Big).
  	\end{eqnarray}
  	
  	\noindent Now, we multiply \eqref{7e11} by $p_{n,j}$ and then sum over $j=1$ to $n$ to get
  	\begin{eqnarray}\label{7e12}
  	\sum_{j=1}^{n} \, p_{n,j} \, ^{c}_{0}{D}_{t_j} v & \le & \frac{11}{4} \, \sum_{j=1}^{n} \, p_{n,j} \, \sum_{k=1}^{j} \, g_{j,k} \, \Big( \int_{t_{k-1}}^{t_k} v'(\eta) \, d\eta \Big) \nonumber \\
  	& \le & \frac{11}{4} \, \sum_{k=1}^{n} \, \Big( \int_{t_{k-1}}^{t_k} v'(\eta) \, d\eta \Big) \, \sum_{j=k}^{n} \, p_{n,j} \, g_{j,k} \nonumber \\ 
  	& \le & \frac{11}{4} \, \int_{t_0}^{t_n} v'(\eta) \, d\eta.
  	\end{eqnarray}
  	
  	\textit{(2)} Assume that $v'$ is nonnegative and monotonic. Suppose $v'$ is monotonic decreasing. As $^{c}_{0}{D}_{t_j} v \ge 0,$ \eqref{7e5a} follows from \textit{(1)}.\\
  	If $v'$ is monotonic increasing, then
  	\begin{eqnarray}\label{7e17}
  	 \sum_{j=1}^{n-1} \, p_{n,j} \, ^{c}_{0}{D}_{t_j} v & = & \sum_{j=1}^{n-1} \, p_{n,j} \, \sum_{s=1}^{m} \, \mu_s \, ^{c}_{0}{D}^{\alpha_s}_{t_j} v \nonumber \\
  	 & = & \sum_{j=1}^{n-1} \, p_{n,j} \, \sum_{s=1}^{m} \, \mu_s \, \sum_{k=1}^{j} \int_{t_{k-1}}^{t_k} \frac{(t_j - \eta)^{- \alpha_s}}{\Gamma{(1- \alpha_s)}} \, v'(\eta) \, d\eta \nonumber \\
  	 & \le & \sum_{j=1}^{n-1} \, p_{n,j} \, \sum_{k=1}^{j} \, v'(t_k) \, \sum_{s=1}^{m} \, \mu_s \int_{t_{k-1}}^{t_k} \frac{(t_j - \eta)^{- \alpha_s}}{\Gamma{(1- \alpha_s)}} \,  d\eta \nonumber \\
  	 & \le & \frac{11}{4} \, \sum_{j=1}^{n-1} \, p_{n,j} \, \sum_{k=1}^{j} \, v'(t_k) \, g_{j,k} \, \tau_k \nonumber \\
  	 & \le & \frac{11}{4} \, \sum_{k=1}^{n-1} \, v'(t_k) \, \tau_k \, \sum_{j=k}^{n-1} \, p_{n,j} \, g_{j,k}  \nonumber \\
%  	 & \le & \frac{11}{4} \, \sum_{k=1}^{n-1} \, v'(t_k) \, \tau_k \nonumber \\
  	 & \le & \frac{11}{4} \, \sum_{k=1}^{n-1} \, v'(t_k) \, \tau_{k+1} \nonumber \\
  	 & \le & \frac{11}{4} \, \sum_{k=1}^{n-1} \, \int_{t_k}^{t_{k+1}} v'(\eta) \, d\eta \nonumber \\
  	 & \le & \frac{11}{4} \, \int_{0}^{t_n} v'(\eta) \, d\eta. 
  	\end{eqnarray}
This completes the proof. 	 
\end{proof}

\begin{lemma}\label{7l4} 
   For any constant $\mathfrak{K} \in \mathbb{R}^{+}$,
  \begin{eqnarray}\label{7e18}
  \sum_{j=1}^{n-1} \, p_{n,j} \, E_{\alpha_1} (\mathfrak{K} \, t_j^{\alpha_1}) & \le & \frac{11}{4 \, \mu_1 \, \mathfrak{K}} \, \Big( E_{\alpha_1} (\mathfrak{K} \, t_n^{\alpha_1}) -1 \Big),
  \end{eqnarray}
  where $E_{\alpha_1}$ is the Mittag-Leffler function.
\end{lemma}

\begin{proof}
  Let $\mathfrak{K} \in \mathbb{R}^{+}$ and $v_k(t) := \frac{t^{k \alpha_1}}{\Gamma {(1+ k \alpha_1)}}$. From the definition of the Mittag-Leffler function,
  \begin{eqnarray}\label{7e19}
  E_{\alpha_1} (\mathfrak{K} \, t^{\alpha_1}) & = & \sum_{k=0}^{\infty} \, \frac{\mathfrak{K}^k t^{k \alpha_1}}{\Gamma {(1+ k \alpha_1)}} \; = \; 1 +  \sum_{k=0}^{\infty} \, \mathfrak{K}^k \, v_k(t).
  \end{eqnarray}
  
  \noindent Since $v'$ is nonnegative and monotonic, it follows from Lemma~\ref{7l2}\textit{(2)} that 
  \begin{eqnarray}\label{7e25}
  \sum_{j=1}^{n-1} \, p_{n,j} \, ^{c}_{0}{D}_{t_j} v_k & \le & \frac{11}{4} \, \int_{0}^{t_n} v_k'(\eta) \, d\eta \; = \; \frac{11}{4} \, v_k (t_n). 
  \end{eqnarray}
  
  \noindent As $v'$ is nonnegative, Also, $ ^{c}_{0}{D}^{\alpha_1}_{t_j} v, \: ^{c}_{0}{D}_{t_j} v \ge 0$. Also, $^{c}_{0}{D}^{\alpha_1}_{t_j} v_k = v_{k-1} (t_j)$. Therefore,  
  \begin{eqnarray}\label{7e26}
  \mu_1 \, \sum_{j=1}^{n-1} \, p_{n,j} \, v_{k-1} (t_j)  \, \le \, \sum_{j=1}^{n-1} \, p_{n,j} \, ^{c}_{0}{D}_{t_j} v_k \, \le \, \frac{11}{4} \, v_k (t_n). 
  \end{eqnarray}
  
  \noindent Now, we multiply \eqref{7e26} with $\mathfrak{K}^k$ and then sum over $k=1$ to $i$ to arrive at 
  \begin{eqnarray}\label{7e27}
   \sum_{k=1}^{i} \, \mathfrak{K}^k \, \sum_{j=1}^{n-1} \, p_{n,j} \, v_{k-1} (t_j) & \le & \frac{11}{4 \, \mu_1} \, \sum_{k=1}^{i} \, \mathfrak{K}^k \, v_k (t_n). 
  \end{eqnarray}
  
  \noindent Therefore,
  \begin{eqnarray}\label{7e29}
   \mathfrak{K} \, \sum_{k=0}^{i-1} \, \sum_{j=1}^{n-1} \, p_{n,j} \, \mathfrak{K}^{k} \, v_k (t_j) & \le & \frac{11}{4 \, \mu_1} \, \Big( \sum_{k=0}^{i} \, \mathfrak{K}^k \, v_k (t_n) -1 \Big). 
  \end{eqnarray}
  
  \noindent After interchanging the sums on L.H.S. of \eqref{7e29} and taking $i \rightarrow \infty$, 
  \begin{eqnarray}\label{7e36}
   \mathfrak{K} \sum_{j=1}^{n-1} \, p_{n,j} \, E_{\alpha_1} (\mathfrak{K} \, t_j^{\alpha_1}) & \le & \frac{11}{4 \, \mu_1 \, \mathfrak{K}} \, \Big( E_{\alpha_1} (\mathfrak{K} \, t_n^{\alpha_1}) -1 \Big). 
  \end{eqnarray}
  This completes the proof.
\end{proof}

%%%% Gronwall inequality
 In the following theorem, now we write the discrete fractional Gr$\ddot{{o}}$nwall inequality, which is a multi-term analogue of~\cite[Lemma 4.1]{[r16]}.
 
\begin{theorem}\label{7l5}  
	 Let $\lambda_i$ be nonnegative constants with $0 < \sum_{i=0}^{N-1} \lambda_i \le \Lambda$ for $n \ge 1$, where $\Lambda$ is some positive constant independent of $n$. Suppose that the nonnegative sequences $\{\xi ^n, \: \zeta ^n : n \ge 1\}$ are bounded and the nonnegative grid function $\{v^n : n \ge 0\}$ satisfies 
	\begin{eqnarray}\label{7e105}
	\begin{split}
	D_N (v^{n- \sigma_n})^2 & \le & \sum_{k=1}^{n} \, \lambda_{n-k} \, (v^{k, \sigma_k})^2 + \, \xi^n \, v^{k, \sigma_k} \, + \,  (\zeta ^n)^2, \quad \mbox{for} \; \, 1 \le n \le N.
	\end{split}
	\end{eqnarray} 
	
	\noindent If the non-uniform grid satisfies the criterion $\max\limits_{1 \le n \le N}\tau_n \le \big(\frac{11}{2} \, \mu \, \Lambda \big)^{-\frac{1}{\alpha_1}}$,  then  
	\begin{eqnarray}\label{7e109}
	v^n & \le & 2 E_{\alpha} \bigg(\frac{11}{2 \, \mu_1} \, \Lambda \, t_n^{\alpha}\bigg) \, \Big[ v^0 + \max_{1 \le k \le n} \sum_{j=1}^{k} \, p_{k,j} \, (\xi^j + \zeta^j) + \max_{1 \le j \le n} \, \{ \zeta^j \} \Big].
	\end{eqnarray}	
\end{theorem}

\begin{proof}
  The proof of this theorem uses Lemmas~ \ref{7l2} and \ref{7l4}, and the principle of mathematical induction. One can prove it using the similar arguments given in the proof for~\cite[Lemma 4.1]{[r16]}.	
\end{proof}

%%%%%
\section{A priori Bound}\label{7bound}
To derive \emph{a priori} bound for the fully-discrete solution $U^n$, first we write two Lemmas, which we will use in further analysis.

\begin{lemma}\label{7l1}
	 For any function $v = v( \cdot , \, t)$ defined on a graded mesh $\left\lbrace  t_j \right\rbrace ^{N}_{j=0}$, one has
	\begin{eqnarray}\label{7ee12}
	   \frac{1}{2} \, {D}_{N} \, \|v^{n + \sigma_n}\|^2 & \le & \left( {D}_{N} \, v^{n- \sigma_n}, \, v^{n, \sigma_n} \right), \quad \mbox{for} \; n= 1, \ldots, N.
	\end{eqnarray}
\end{lemma}

\begin{proof} Proof follows form the following result given in \cite[Corollary 1]{[fmr13]}. For any $\alpha, \sigma \in (0,1),$
\begin{eqnarray}\label{7ee12a}
	\frac{1}{2} \, {D}^{\alpha}_{N} \|v^{n + \sigma_n}\|^2 & \le & \left( {D}^{\alpha}_{N} v^{n- \sigma_n}, \, v^{n, \alpha} \right), \quad \mbox{for} \; n= 1, \ldots, N. \nonumber
\end{eqnarray}
\end{proof}

\begin{lemma}\label{7l7}  
	Let $\gamma_s \in (0,1), \forall s=1,...,m.$ Then the coefficients $p_{n,i}$ defined in \eqref{7e121} satisfy the following estimate. 
	\begin{eqnarray}\label{7e124a}
	\sum_{i=1}^{n} \Big( \sum_{s=1}^{m} \mu_s \, \frac{\Gamma{(1+ \gamma_s)}}{\Gamma{(1+ \gamma_s - \alpha_s)}} \, t_i^{\gamma_s - \alpha_s} \Big) \, p_{n,i} & \le & \frac{11}{4} \, \sum_{s=1}^{m} \, t_n^{\gamma_s}.
	\end{eqnarray}
	%	In particular, if $\gamma_s = \alpha_s$,
	%	\begin{eqnarray}\label{7e124}
	%	   \sum_{i=1}^{n} p_{n,i} & \le & \frac{11\, T^{\alpha_1}}{4\, \min\limits_{1 \le s \le m} (\mu_s \Gamma{(1+ \alpha_s)})} .
	%	\end{eqnarray}	
\end{lemma} 

\begin{proof}
	Let $\gamma_j \in (0,1), \forall j=1,...,m.$
	To prove this result, we define the function $\mathcal{V}(t) := \sum_{j=1}^{m} t^{\gamma_j}$. 
	\begin{eqnarray}\label{7e143}
	\sum_{s=1}^{m} \mu_s \, \frac{\Gamma{(1+ \gamma_s)}}{\Gamma{(1+ \gamma_s - \alpha_s)}} \, t_i^{\gamma_s - \alpha_s} & \le & \sum_{s=1}^{m} \mu_s \Big(\sum_{j=1}^{m} \frac{\Gamma{(1+ \gamma_j)}}{\Gamma{(1+ \gamma_j - \alpha_s)}} \, t_i^{\gamma_j - \alpha_s} \Big) \nonumber  \vspace{2pt}\\
	& = &  ^{c}_{0}{D}_{t_{i}} \mathcal{V}.
	\end{eqnarray}
	
	\noindent Multiplying \eqref{7e143} by $ p_{n,i}$ and then summing from $i=1$ to $n$, we have 
	\begin{eqnarray}\label{7e144}
	\sum_{i=1}^{n} \Big( \sum_{s=1}^{m} \mu_s \, \frac{\Gamma{(1+ \gamma_s)}}{\Gamma{(1+ \gamma_s - \alpha_s)}} \, t_i^{\gamma_s - \alpha_s} \Big) \, p_{n,i} & \le & \sum_{i=1}^{n} \, ^{c}_{0}{D}_{t_{i}} \mathcal{V} \,  p_{n,i}.
	\end{eqnarray}
	
	\noindent Now, we use Lemma \ref{7l2} into \eqref{7e144} to arrive at 
	\begin{eqnarray}\label{7e145}
	\sum_{i=1}^{n} \Big( \sum_{s=1}^{m} \mu_s \, \frac{\Gamma{(1+ \gamma_s)}}{\Gamma{(1+ \gamma_s - \alpha_s)}} \, t_i^{\gamma_s - \alpha_s} \Big) \, p_{n,i} & \le &  \frac{11}{4} \int_{0}^{t_n} \sum_{j=1}^{m} \gamma_j \, \eta^{\gamma_j -1} \, d\eta \nonumber \\
	& = &  \frac{11}{4} \,  \sum_{s=1}^{m} t_n^{\gamma_s}.
	\end{eqnarray}
	This completes the proof. 
\end{proof}

%%% A Priori bound
In the following theorem, we derive \textit{a priori} bound for the full-discrete solution.
\begin{theorem}\label{2th4}
	For each  $n=1, \ldots, N$, the fully-discrete solution $U^{n}$ of \eqref{7e6} satisfies
	\begin{eqnarray}
	\label{2e101}
	\|U^{n}\| & \le & C \, \big(1+\|U^0\| \big), \\
	\label{2e102}
	\|\nabla U^{n}\| & \le & C  \, \big(1 + \|\nabla U^0\| \big).
	\end{eqnarray}	
\end{theorem}
	
\begin{proof} 
By choosing $v_h= U^{n, \sigma_n}$ in \eqref{7e6:1.2} and \eqref{7e6:1.3}, and then using the Hypothesis H1, Cauchy-Schwartz inequality, Lemma \ref{7l1} and the Young's inequality, we get
\begin{eqnarray}\label{2e13}
\frac{1}{2} \, {D}_{N} \|U^{n - \sigma_n}\|^2 & \le & \frac{1}{2} \, \|f^{n-\sigma_n}\|^2 + \frac{1}{2} \, \|U^{n, \sigma_n}\|^2.
\end{eqnarray}

\noindent
An application of Lemma \ref{7l5} into \eqref{2e13} leads to	
\begin{eqnarray}\label{2e14}
\|U^{n}\|  & \le & 2  E_{\alpha}\Big(\frac{11}{2 \, \mu_1} \, t_{n}^{\alpha_1}\Big) \Big( \|U^0\|  +  \! \max_{1 \le k \le n} \! \Big( \max_{1 \le j \le k} \|f^{j-\sigma_n}\|\Big) \sum_{j=1}^{k} p_{k,j} + \! \max_{1 \le j \le n}  \|f^{j-\sigma_n}\| \Big). \nonumber
\end{eqnarray}

\noindent
Thus, using Lemma \ref{7l7} with $\gamma_s = \alpha_s$ in the above inequality, we have 	
\begin{eqnarray}\label{2e16}
\|U^{n}\|  & \le & C \, \Big( 1 + \|U^0\| \Big), \nonumber
\end{eqnarray}
where $C = 2 E_{\alpha}\bigg(\frac{11}{2 \, \mu_1} \, t_{n}^{\alpha_1}\bigg) \, \max \bigg\{ 1, \,  \Big( \frac{11\, T^{\alpha_1}}{4\, \min\limits_{1 \le s \le m} (\mu_s \Gamma{(1+ \alpha_s)})} + 1\Big) \max\limits_{1 \le j \le n} \|f^{j-\sigma_n}\| \bigg\}$.\\

\noindent
To derive an estimate for $\|\nabla U^{n}\|$, we rewrite \eqref{7e6:1.2} and \eqref{7e6:1.3} using the operator $\Delta_h$~\cite[P.10]{[vth]}, and then we take $v_h=\, - \Delta_h U^{n, \sigma_n}$ and proceed in a similar manner as above to get
\begin{eqnarray}\label{2e86}
\|\nabla U^{n}\|  & \le & C \, \Big( 1 + \|\nabla U^0\| \Big), \nonumber
\end{eqnarray}
where $C = 2 \, E_{\alpha}\Big(\frac{11}{2 \, \mu_1} t_{n}^{\alpha}\Big) \, \max \left\lbrace 1, \, \frac{1}{\sqrt{m_1}} \Big( \frac{11\, T^{\alpha_1}}{4\, \min\limits_{1 \le s \le m} (\mu_s \Gamma{(1+ \alpha_s)})} + 1\Big) \max\limits_{1 \le j \le n} \|f^{j-\sigma_n}\| \right\rbrace$. \\
The proof is complete.
\end{proof}

\begin{remark}\label{7rmk3}
	The existence-uniqueness result for the fully-discrete solution $U^n$ of \eqref{7e6} follows, in similar lines as given in~\cite[Theorem 4.1]{[sp1]}, by using \textit{a priori} bound and a consequence of the Brouwer's fixed point theorem \cite[P.237]{[vth]}
\end{remark}
%%%%%%%%%%%%%%%%%%%%%%%%%%%%%%%%%%%%%%%%%%%%%%%%%%%%%%%%%%%%%%%%%%%%%%
\section{Error Estimates}\label{7err_est}
In order to derive \emph{a priori} error estimates, we require some additional regularity on solution $u$. Therefore, we assume that solution $u$ satisfies the following conditions.  
\begin{eqnarray}\label{7e60}
&{^{c}_{0}{D}_{t} \, u}& \! \! \! \in {L^{\infty}(0, T; {H^2(\Omega)})}, \qquad   u \in {L^{\infty}(0, T; {H^1_0(\Omega) \cap H^3(\Omega)})}, \nonumber \\
&\! \! \! \mbox{and}&  \, \| \partial_t^{q} u\|_3 \le C \, (1+t^{\alpha_1 - q}), \quad \mbox{for} \;\, q = 0,1,2,3. 
\end{eqnarray}

The following lemmas and definition will be useful to derive error estimates.

\begin{lemma}\label{7l3}
	\cite[P.9 (32)]{[hr12]} If $u \in C^1[0,T] \cap C^2(0,T]$ and $\| \partial_t^{q} u\| \le C \, (1+t^{\alpha_1 - q})$, for $q = 0,1,2,3$, then, for $n=1, \ldots, N,$ 	
	\begin{eqnarray}\label{7e123}
	 t_{n-\sigma_n}^{\alpha_1} \, \|u^{n, \sigma_n} - u^{n-\sigma_n} \| & \le & C N^{-\min\{r \alpha_1, \, 2\}}. \nonumber
	\end{eqnarray}
\end{lemma}

\begin{lemma}\label{7l8}
	 If $u \in C^1[0,T] \cap C^2(0,T]$ and $\| \partial_t^{q} u\| \le C \, (1+t^{\alpha_1 - q})$, for $q = 0,1,2,3$, then, for $n=2, \ldots, N,$ 	
	\begin{eqnarray}\label{7e124}
	  \|u^{n - \sigma_n} - \widehat{u}^{n - \sigma_n} \| & \le & C N^{-\min\{r \alpha_1, \, 2\}}. \nonumber
	\end{eqnarray}
\end{lemma}

\begin{proof}
	The proof follows from the Taylor's theorem, and the estimates \eqref{7e20} and \eqref{7e20a}.
\end{proof}

\begin{definition}\cite[P.8]{[vth]}
	The Ritz-projection is a map  $R_h : H^1_0(\Omega) \rightarrow V_h$ such that
	\begin{equation*}\label{7e62}
	(\nabla w, \, \nabla v) \, = \, (\nabla R_h w, \, \nabla v), \quad \forall w \in  H^1_0(\Omega) \: \mbox{and}  \; \;  \forall v \in V_h.
	\end{equation*}
\end{definition}

\indent Utilizing \eqref{7e60}, one can easily prove that $R_h $ satisfies $\|{\nabla R_h u}\| \, \le \, C$. \\

\indent Since $V_h$ is the space of linear polynomials, it is well known \cite[Lemma 1.1]{[vth]} that
\begin{equation}\label{7e61}
\begin{split}
\|{w-R_hw} \| \, + h \, \|{\nabla (w-R_hw)}\| \, \le& \, Ch^2 \, \|{\Delta w}\|, \quad \forall w \in H^2 \cap H^1_0.\\
\end{split}
\end{equation}

The next theorem is one of the main results of this article. In this theorem, we state and prove the convergence results for the fully-discrete solution. For this, we split the error ($u^n - U^n$) as follows:
\begin{equation}\label{7e61n}
\begin{split}
u^n - U^n = u^n - R_hu^n + R_hu^n - U^n = \rho ^n + \theta ^n,
\end{split}
\end{equation}
where \  $ \rho ^n := u^n - R_hu^n$ \ and \ $\theta ^n := R_hu^n - U^n.$\\

%%% Error estimate
\begin{theorem}\label{7th5}
	For $1 \le n \le N$, let $u^{n}$ and $U^{n}$ be the solutions of \eqref{7e1} and \eqref{7e6}, respectively. Then
	\begin{eqnarray}
	\label{7e21a}
	\max_{1 \le n \le N} \|u^{n} - U^{n}\|& \le & C \, \big(h^2 + N^{-\min \left\lbrace 2, \, r \alpha \right\rbrace }\big), \\
	\label{7e21}
	\max_{1 \le n \le N} \|\nabla (u^{n} - U^{n}) \| & \le & C \, \big(h + N^{-\min \left\lbrace 2, \, r \alpha \right\rbrace }\big).
	\end{eqnarray}
\end{theorem}

\begin{proof}
  For any $v_h \in V_h$, the estimate for $\theta ^n$ is given by
  \begin{eqnarray}\label{7e22}
  \big( {D}_{N} \theta^{1-\sigma_1}, \! \! \! \! &v_h& \! \! \! \! \! \! \! \big) \, + \, a\big(l(U^{1, \sigma_1})\big) \, (\nabla \theta^{1, \sigma_1}, \nabla v_h) \nonumber\\
  & = & \! \! \! \! \! \big( {D}_{N} R_hu^{1-\sigma_1} - \, ^{c}_{0}{D}_{t_{1-\sigma_1}}u , v_h \big) - a\big(l(u^{1-\sigma_1})\big) (\Delta u^{1, \sigma_1} - \Delta u^{1-\sigma_1}, v_h) \nonumber \\
  & \quad - & \! \! \! \!  \big\{ a\big(l(U^{1, \sigma_1})\big) - a\big(l(u^{1-\sigma_1})\big) \big\} (\Delta u^{1, \sigma_1}, v_h),
  \end{eqnarray}
  
  \noindent and for $n \ge 2$,
  \begin{eqnarray}\label{7e23}
  \big( {D}_{N} \theta^{n-\sigma_n}, \! \! \! \! &v_h& \! \! \! \! \! \! \! \big) \, + \, a\big(l(\widehat{U}^{n - \sigma_n})\big) \, (\nabla \theta^{n, \sigma_n}, \nabla v_h) \nonumber\\
  & = & \! \! \! \! \! \big( {D}_{N} R_hu^{n-\sigma_n} - \, ^{c}_{0}{D}_{t_{n-\sigma_n}}u , v_h \big) - a\big(l(u^{n-\sigma_n})\big) (\Delta u^{n, \sigma_n} - \Delta u^{n-\sigma_n}, v_h) \nonumber \\
  & \quad - & \! \! \! \!  \big\{ a\big(l(\widehat{U}^{n - \sigma_n})\big) - a\big(l(u^{n-\sigma_n})\big) \big\} (\Delta u^{n, \sigma_n}, v_h).
  \end{eqnarray}

  \noindent Also, for $n \ge 2$, Lipschitz continuity of the function $a$ gives
   \begin{eqnarray}\label{7e28}
   \big\| a\big(l(\widehat{U}^{n - \sigma_n})\big) - a\big(l(u^{n-\sigma_n})\big) \big\| & \le & \| \widehat{U}^{n - \sigma_n} - u^{n-\sigma_n} \| \nonumber \\
   & \le & \| \widehat{U}^{n - \sigma_n} - \widehat{u}^{n-\sigma_n} \| + \| \widehat{u}^{n-\sigma_n} - u^{n-\sigma_n} \| \nonumber \\
   & \le & C \, \big(h^2  + \| \theta^{n-1, \sigma_{n-1}}\| + N^{-\min\{r \alpha_1, \, 2\}} \big), 
   \end{eqnarray}
   where we have used Lemma~\ref{7l8}. \\
   
   \noindent Now, we choose $v_h = \theta^{1, \sigma_1}$ and $v_h = \theta^{n, \sigma_n}$ in \eqref{7e22} and \eqref{7e23}, respectively to get
   \begin{eqnarray}\label{7e22a}
   \big( {D}_{N} \theta^{1-\sigma_1}, \! \!  &\theta^{1, \sigma_1}& \! \! \! \! \! \! \big) \, + \, a\big(l(U^{1, \sigma_1})\big) \, (\nabla \theta^{1, \sigma_1}, \nabla \theta^{1, \sigma_1}) \nonumber\\
   & = \quad &  \! \! \! \! \! \! \! \! \! \! \big( {D}_{N} R_hu^{1-\sigma_1} - \, ^{c}_{0}{D}_{t_{1-\sigma_1}}u , \theta^{1, \sigma_1} \big) - a\big(l(u^{1-\sigma_1})\big) (\Delta u^{1, \sigma_1} - \Delta u^{1-\sigma_1}, \theta^{1, \sigma_1}) \nonumber \\
   & & \;  -  \, \big\{ a\big(l(U^{1, \sigma_1})\big) - a\big(l(u^{1-\sigma_1})\big) \big\} (\Delta u^{1, \sigma_1}, \theta^{1, \sigma_1}),
   \end{eqnarray}
   
   \noindent and for $n \ge 2$,
    \begin{eqnarray}\label{7e23a}
    \big( {D}_{N} \theta^{n-\sigma_n}, \! \!  &\theta^{n, \sigma_n}& \! \! \! \! \! \!  \big) \, + \, a\big(l(\widehat{U}^{n - \sigma_n})\big) \, (\nabla \theta^{n, \sigma_n}, \nabla \theta^{n, \sigma_n}) \nonumber\\
   & = \quad &  \! \! \! \! \! \! \! \! \! \! \big( {D}_{N} R_hu^{n-\sigma_n} - \, ^{c}_{0}{D}_{t_{n-\sigma_n}}u , \theta^{n, \sigma_n} \big) - a\big(l(u^{n-\sigma_n})\big) (\Delta u^{n, \sigma_n} - \Delta u^{n-\sigma_n}, \theta^{n, \sigma_n}) \nonumber \\
    & & \;  -  \,  \big\{ a\big(l(\widehat{U}^{n - \sigma_n})\big) - a\big(l(u^{n-\sigma_n})\big) \big\} (\Delta u^{n, \sigma_n}, \theta^{n, \sigma_n}).
    \end{eqnarray}
    
    \noindent Using the Hypotheses~H1-H2, the Cauchy-Schwartz inequality, \eqref{7e60}, \eqref{7e61}, Lemma~\ref{7l1}, we have
    \begin{eqnarray}\label{7e24}
      \! \! \! \! \! {D}_{N} \|\theta^{1-\sigma_1}\|^2 \le C \, \|\theta^{1, \sigma_1}\|^2 +  C \Big\{ \|{D}_{N} R_hu^{1-\sigma_1} \! \! \! \!  &-& \! \! \! ^{c}_{0}{D}_{t_{1-\sigma_1}}R_hu \| + \|u^{1, \sigma_1} - u^{1-\sigma_1}\| \nonumber \\
     & + & \! \! \!  \|\Delta u^{1, \sigma_1} - \Delta u^{1-\sigma_1}\| +h^2 \Big\} \|\theta^{1, \sigma_1}\|. \quad
    \end{eqnarray}
    
    \noindent Using the same results together with estimate \eqref{7e28}, for $n \ge 2$, one can obtain
    \begin{eqnarray}\label{7e30}
    {D}_{N} \|\theta^{n-\sigma_n}\|^2 & \le & C \, \big\{ \|\theta^{n, \sigma_n}\|^2 + \| \theta^{n-1, \sigma_{n-1}}\|^2 \big\}  + C \, \Big\{ \|{D}_{N} R_hu^{n-\sigma_n} - \, ^{c}_{0}{D}_{t_{n-\sigma_n}}R_hu \| \nonumber \\
    & & \hspace{1cm} + \; \|\Delta u^{n, \sigma_n} - \Delta u^{n-\sigma_n}\| +  \; h^2 + N^{-\min\{r \alpha_1, \, 2\}} \Big\} \, \|\theta^{n, \sigma_n}\|.       
    \end{eqnarray}
        
    \noindent Moreover, from Theorem~\ref{7l6} and Lemma~\ref{7l3}, it follows that
    \begin{eqnarray}\label{7e32}
     & & \|{D}_{N} R_hu^{n-\sigma_n} - \, ^{c}_{0}{D}_{t_{n-\sigma_n}}R_hu \| + \|\Delta u^{n, \sigma_n} - \Delta u^{n-\sigma_n}\|  \nonumber \\ 
     & & \hspace{1.5cm} \le  \, C \, N^{-\min\{r \alpha_1, \, 2\}} \, \Big\{\sum_{s=1}^{m} \frac{\mu_s}{\Gamma{(1-\alpha_s)}} \, t^{-\alpha_s}_{n - \sigma_n} + \, t_{n-\sigma_n}^{- \alpha_1} \Big\}. \nonumber \\
     & & \hspace{1.5cm} \le  \, C \, N^{-\min\{r \alpha_1, \, 2\}} \, \Big\{\sum_{s=1}^{m} \mu_s \, t^{-\alpha_s}_{n - \sigma_n} \Big\}. \nonumber \\
     & & \hspace{1.5cm} \le  \, C \, N^{r \, l_N} \, N^{-\min\{r \alpha_1, \, 2\}} \, \Big\{\sum_{s=1}^{m} \mu_s \, t^{l_N -\alpha_s}_{n} \Big\}. \nonumber \\
      & &  \hspace{1.5cm} \le  \, C_1 \, N^{-\min\{r \alpha_1, \, 2\}} \, \Big\{\sum_{s=1}^{m} \mu_s \, \frac{\Gamma(1 + l_N)}{\Gamma(1 + l_N - \alpha_s)} \, t^{l_N -\alpha_s}_{n} \Big\},
     \end{eqnarray}
    where $C_1 = C \, N^{r \, l_N} \, \frac{\max\limits_{1 \le s \le m} \Gamma(1 + l_N - \alpha_s)}{\Gamma(1 + l_N)}$.\\
    
    \noindent Similarly, for $n =1$, we can have 
    \begin{eqnarray}\label{7e33}
    & & \|{D}_{N} R_hu^{1-\sigma_1} - \, ^{c}_{0}{D}_{t_{1-\sigma_1}}R_hu \| +  \|u^{1, \sigma_1} - u^{1-\sigma_1}\| + \|\Delta u^{1, \sigma_1} - \Delta u^{1-\sigma_1}\|  \nonumber \\ 
    & &  \hspace{1.5cm} \le  \, C_1 \, N^{-\min\{r \alpha_1, \, 2\}} \, \Big\{\sum_{s=1}^{m} \mu_s \, \frac{\Gamma(1 + l_N)}{\Gamma(1 + l_N - \alpha_s)} \, t^{l_N -\alpha_s}_{1} \Big\}.
    \end{eqnarray}

    \noindent Now, applying Theorem~\ref{7l5} into \eqref{7e24} and \eqref{7e30}, and then using  \eqref{7e32}, \eqref{7e33} to arrive at
    \begin{eqnarray}\label{7e31} 
    \|\theta^{n}\| & \le & C \, \Big[ C_1 \, N^{-\min\{r \alpha_1, \, 2\}} \,  \max_{1 \le k \le n} \Big\{ \sum_{j=1}^{k} \, \Big(\sum_{s=1}^{m} \mu_s \, \frac{\Gamma(1 + l_N)}{\Gamma(1 + l_N - \alpha_s)} \, t^{l_N -\alpha_s}_{n} \Big) \, p_{k,j} \Big\}  \nonumber \\
    & & \hspace{3cm} + \,  \|\theta^0\| \, + \, C_1 \, \big( h^2 + N^{-\min\{r \alpha_1, \, 2\}} \big) \max_{1 \le k \le n} \sum_{j=1}^{k} p_{k,j}  \Big]. \nonumber
    \end{eqnarray}
    
    \noindent We utilize Lemma~\ref{7l7} with $\gamma_s = l_N$ and $\gamma_s = \alpha_s$ into above inequality and choose $U^0 = R_h u^0$ to obtain    \begin{eqnarray}\label{7e34} 
    \|\theta^{n}\| & \le & C \,  \big( \|\theta^0\| + h^2 + N^{-\min\{r \alpha_1, \, 2\}} \big).
    \end{eqnarray}
    
    \noindent In the end, application of the Triangle inequality, and  the estimates \eqref{7e61}, \eqref{7e34} yields \eqref{7e21a}.
    
    \noindent  In order to prove the result for $H_0^1$-norm, we rewrite \eqref{7e22} and \eqref{7e23} using the definition of $\Delta_h$~\cite[P.10]{[vth]} and then take $v_h=\, - \Delta_h \theta^{n, \sigma_n}$. After that we use similar arguments as above together with the Poincar$\acute{e}$ inequality to get \eqref{7e21}.

\end{proof}

%%%%%%% 
\section{Numerical Experiments}\label{7NE}
In this section, we perform two different numerical experiments with known exact solution to confirm our theoretical estimates. In both numerical experiments, we consider $[0,1]$ as the time interval and we take $\mu_s=1,$ ($\forall \, s=1, \ldots, m$). In order to achieve optimal order of convergence rate, we choose mesh grading parameter $r=\frac{2}{\alpha_1}.$ Spatial domain $\Omega$ is partitioned with $\big(M_s+1\big)$ node points in each direction and the time interval is partitioned into $N$ number of sub-intervals. We set $M_s=N$ and calculate error for different values of $N$ to get the rate of convergence in time direction. Similarly, We set $N=M_s$ and calculate error for different values of $M_s$ to conclude the rate of convergence in spatial direction. \\

\begin{example}\label{7E2}
In our first example, we consider \eqref{7e1} with spatial domain $\Omega = (0, \pi)$, $a(w)=3+\sin w$ and $f(x,t)$ is chosen such that the analytical solution of given PDE is $u(x,t)=(t^3+t^{\alpha_1}) \sin x$. 	
\end{example}

\noindent Table~\ref{7T5} shows the order of convergence in the temporal direction in $L^{\infty}(0,T;L^2(\Omega))$ norm on a graded mesh. From this table, it can be seen that using graded mesh, we can achieve $O(N^{-2})$ convergence rate in the temporal direction. The order of convergence in the spatial direction in $L^2(\Omega)$ and $H^1(\Omega)$ norms are given in Tables~\ref{7T6}~and~\ref{7T7}, respectively. From these, one can observe that the numerical convergence results are in accordance to our theoretical convergence estimates. 

\begin{center}
	\begin{tiny}
		\begin{table}[h!]
			\begin{center}
				\begin{tabular}{|c|c|c|c|c|}
					\hline
					%%%  FOR COLOR FONT
					%								\multirow{2}{*}{\textcolor{red}{\large $N$}} &        \multicolumn{2}{c|}{\textcolor{red}{$\alpha = 0.4$}} & \multicolumn{2}{c|}{\textcolor{red}{$\alpha = 0.5$}} & \multicolumn{2}{c|}{\textcolor{red}{$\alpha = 0.7$}}\\
					%								\cline{2-7}
					%								&  & \textcolor{plum}{Error} & \textcolor{plum}{OC} & \textcolor{plum}{Error} & \textcolor{plum}{OC} & \textcolor{plum}{Error} & \textcolor{plum}{OC}  \\	
					%%%  FOR BLACK FONT
					\multirow{3}{*}{\large $N$} & \multicolumn{2}{c|}{$\alpha_1 = 0.4$, $\alpha_2 = 0.37$} & \multicolumn{2}{c|}{$\alpha_1 = 0.7$, $\alpha_2 = 0.68$, $\alpha_3 = 0.66$} \\
					& \multicolumn{2}{c|}{$\alpha_3 = 0.35$, $\alpha_4 = 0.33$} & \multicolumn{2}{c|}{$\alpha_4 = 0.64$, $\alpha_5 = 0.62$} \\
					\cline{2-5}
					& Error & OC & Error & OC  \\
					%%%	
					\hline
					$2^6$ & 1.50E-02 & 1.918073 & 2.42E-03 & 1.978989  \\
					\hline
					$2^7$ & 3.97E-03 & 1.967703 & 6.14E-04 &  1.993808 \\
					\hline
					$2^8$ & 1.01E-03 & 1.982828 & 1.54E-04 & 2.000700 \\
					\hline
					$2^9$ & 2.56E-04 & - & 3.85E-05 & - \\
					\hline 
			\end{tabular}
			\end{center}
			\caption {\emph {Error and order of convergence  in $L^{\infty}(0,T;L^2(\Omega))$ norm in temporal direction on a graded mesh for Example \ref{7E2}.}}
			\label{7T5}
		\end{table}
	\end{tiny}
\end{center}
%%%

\begin{center}
	\begin{tiny}
		\begin{table}[h!]
			\begin{center}	
				\begin{tabular}{|c|c|c|c|c|}
					\hline
					%%%  FOR COLOR FONT
					%				\multirow{2}{*}{\textcolor{red}{\large $M$}} &    \multicolumn{3}{c|}{\textcolor{red}{$\alpha = 0.4$}} & \multicolumn{3}{c|}{\textcolor{red}{$\alpha = 0.5$}} & \multicolumn{3}{c|}{\textcolor{red}{$\alpha = 0.7$}}\\
					%				\cline{2-10}
					%				 & \textcolor{plum}{Error} & \textcolor{plum}{OC}  & \textcolor{plum}{Error} & \textcolor{plum}{OC} & \textcolor{plum}{Error} & \textcolor{plum}{OC}  \\
					%%%  FOR BLACK FONT
					\multirow{3}{*}{\large $M_s$} & \multicolumn{2}{c|}{$\alpha_1 = 0.4$, $\alpha_2 = 0.37$} & \multicolumn{2}{c|}{$\alpha_1 = 0.7$, $\alpha_2 = 0.68$, $\alpha_3 = 0.66$} \\
					& \multicolumn{2}{c|}{$\alpha_3 = 0.35$, $\alpha_4 = 0.33$} & \multicolumn{2}{c|}{$\alpha_4 = 0.64$, $\alpha_5 = 0.62$} \\
					\cline{2-5}
					& Error & OC & Error & OC  \\
					%%%%%%	
					\hline
					$2^6$ & 1.50E-02 & 1.918073 & 2.42E-03 & 1.978989  \\
					\hline
					$2^7$ & 3.97E-03 & 1.967703 & 6.14E-04 & 1.993808 \\
					\hline
					$2^8$ & 1.01E-03 & 1.982828 & 1.54E-04 & 2.000700  \\
					\hline
					$2^9$ & 2.56E-04 & - & 3.85E-05 & -  \\
					\hline 
				\end{tabular}
				\caption {\emph{Error and order of convergence in $L^2(\Omega)$ norm in space for Example \ref{7E2}.}}
				\label{7T6}
			\end{center}
		\end{table}
	\end{tiny}
\end{center}
%%%

\begin{center}
	\begin{tiny}
		\begin{table}[h!]
			\begin{center}
				\begin{tabular}{|c|c|c|c|c|}
					\hline
					%%%  FOR COLOR FONT
					%				\multirow{2}{*}{\textcolor{red}{\large $M$}} &    \multicolumn{3}{c|}{\textcolor{red}{$\alpha = 0.4$}} & \multicolumn{3}{c|}{\textcolor{red}{$\alpha = 0.5$}} & \multicolumn{3}{c|}{\textcolor{red}{$\alpha = 0.7$}}\\
					%				\cline{2-7}
					%				 & \textcolor{plum}{Error} & \textcolor{plum}{OC} &  \textcolor{plum}{Error} & \textcolor{plum}{OC} & \textcolor{plum}{Error} & \textcolor{plum}{OC}  \\
					%%%  FOR BLACK FONT
					\multirow{3}{*}{\large $M_s$} & \multicolumn{2}{c|}{$\alpha_1 = 0.4$, $\alpha_2 = 0.37$} & \multicolumn{2}{c|}{$\alpha_1 = 0.7$, $\alpha_2 = 0.68$, $\alpha_3 = 0.66$} \\
					& \multicolumn{2}{c|}{$\alpha_3 = 0.35$, $\alpha_4 = 0.33$} & \multicolumn{2}{c|}{$\alpha_4 = 0.64$, $\alpha_5 = 0.62$} \\
					\cline{2-5}
					& Error & OC & Error & OC  \\
					%%%%%%	
					\hline
					$2^6$ & 3.84E-02 & 1.077912 & 3.56E-02 & 1.001456  \\
					\hline
					$2^7$ & 1.82E-02 & 1.024152 & 1.78E-02 &  1.000383 \\
					\hline
					$2^8$ & 8.93E-03 & 1.006509 & 8.88E-03 & 1.000097 \\
					\hline
					$2^9$ & 4.45E-03 & - & 4.44E-03 & -  \\
					\hline
				\end{tabular}
				\caption {\emph {Error and order of convergence  in $H^1_0(\Omega)$ norm in space for Example \ref{7E2}.}}
				\label{7T7}
			\end{center}
		\end{table}
	\end{tiny}
\end{center}
%%%%%%%
\newpage
\begin{example}\label{7E3}
In second example, we consider the Equation~\eqref{7e1} with spatial domain $\Omega = (0,1) \times (0,1)$, $a(w)=3+\sin w$ and $f(x,t)$ is chosen such that the analytical solution of given PDE is $u(x,y,t)=(t^3+t^{\alpha_1}) (x-x^2)(y-y^2)$.
\end{example}
\noindent In Table~\ref{7T8}, we have given the order of convergence in the temporal direction in $L^{\infty}(0,T;L^2(\Omega))$ norm on a graded mesh. It can be seen from this table that using graded mesh, we obtain optimal convergence rate in time direction, as predicted by Theorem \ref{7th5}. Tables~\ref{7T9}~and~\ref{7T10} confirm the theoretical convergence estimates in $L^2(\Omega)$ \& $H^1(\Omega)$ norms in spatial direction, respectively. 
\vspace{-0.3cm}
\begin{center}
	\begin{tiny}	
		\begin{table}[h!]
			\begin{center}
				\begin{tabular}{|c|c|c|c|c|}
					\hline
					%%%  FOR COLOR FONT
					%								\multirow{2}{*}{\textcolor{red}{\large $N$}} &        \multicolumn{2}{c|}{\textcolor{red}{$\alpha = 0.4$}} & \multicolumn{2}{c|}{\textcolor{red}{$\alpha = 0.5$}} & \multicolumn{2}{c|}{\textcolor{red}{$\alpha = 0.7$}}\\
					%								\cline{2-7}
					%								&  & \textcolor{plum}{Error} & \textcolor{plum}{OC} & \textcolor{plum}{Error} & \textcolor{plum}{OC} & \textcolor{plum}{Error} & \textcolor{plum}{OC}  \\	
					%%%  FOR BLACK FONT
				\multirow{3}{*}{\large $N$} & \multicolumn{2}{c|}{$\alpha_1 = 0.5$, $\alpha_2 = 0.47$} & \multicolumn{2}{c|}{$\alpha_1 = 0.9$, $\alpha_2 = 0.88$, $\alpha_3 = 0.86$} \\
				& \multicolumn{2}{c|}{$\alpha_3 = 0.45$, $\alpha_4 = 0.43$} & \multicolumn{2}{c|}{$\alpha_4 = 0.84$, $\alpha_5 = 0.82$} \\
				\cline{2-5}
				& Error & OC & Error & OC  \\
					%%%	
					\hline
					$2^4$ & 1.37E-03 & 1.943223 & 1.01E-03 & 1.991199 \\
					\hline
					$2^5$ & 3.56E-04 & 1.976023 & 2.53E-04 & 1.997810 \\
					\hline
					$2^6$ &	9.05E-05 & 1.989217 & 6.35E-05 & 1.999698 \\
					\hline
					$2^7$ & 2.28E-05 & - & 1.59E-05 & - \\
					\hline														
				\end{tabular}
			\end{center}
			\caption {\emph{Error and convergence rate in $L^{\infty}(0,T;L^2(\Omega))$norm in time for Example~\ref{7E3}.}}
			\label{7T8}
		\end{table}
	\end{tiny}
\end{center}
%%%%

\begin{center}
	\begin{tiny}	
		\begin{table}[h!]
			\begin{center}
				\begin{tabular}{|c|c|c|c|c|}
					\hline
					%%%  FOR COLOR FONT
					%								\multirow{2}{*}{\textcolor{red}{\large $N$}} &        \multicolumn{2}{c|}{\textcolor{red}{$\alpha = 0.4$}} & \multicolumn{2}{c|}{\textcolor{red}{$\alpha = 0.5$}} & \multicolumn{2}{c|}{\textcolor{red}{$\alpha = 0.7$}}\\
					%								\cline{2-7}
					%								&  & \textcolor{plum}{Error} & \textcolor{plum}{OC} & \textcolor{plum}{Error} & \textcolor{plum}{OC} & \textcolor{plum}{Error} & \textcolor{plum}{OC}  \\	
					%%%  FOR BLACK FONT
					\multirow{3}{*}{\large $M_s$} & \multicolumn{2}{c|}{$\alpha_1 = 0.5$, $\alpha_2 = 0.47$} & \multicolumn{2}{c|}{$\alpha_1 = 0.9$, $\alpha_2 = 0.88$, $\alpha_3 = 0.86$} \\
					& \multicolumn{2}{c|}{$\alpha_3 = 0.45$, $\alpha_4 = 0.43$} & \multicolumn{2}{c|}{$\alpha_4 = 0.84$, $\alpha_5 = 0.82$} \\
					\cline{2-5}
					& Error & OC & Error & OC  \\
					%%%	
					\hline
					$2^4$ & 1.37E-03 & 1.943223 & 1.01E-03 & 1.991199 \\
					\hline
					$2^5$ & 3.56E-04 & 1.976023 & 2.53E-04 & 1.997810 \\
					\hline
					$2^6$ & 9.05E-05 & 1.989217 & 6.35E-05 & 1.999698 \\
					\hline
					$2^7$ & 2.28E-05 & - & 1.59E-05 & - \\	
					\hline
			\end{tabular}
			\end{center}
			\caption {\emph {Error and order of convergence in $L^2(\Omega)$ norm in space for Example \ref{7E3}.}}
			\label{7T9}
		\end{table}
	\end{tiny}
\end{center}
%%%

\begin{center}
	\begin{tiny}	
		\begin{table}[h!]
			\begin{center}
				\begin{tabular}{|c|c|c|c|c|}
					\hline
					%%%  FOR COLOR FONT
					%								\multirow{2}{*}{\textcolor{red}{\large $N$}} &        \multicolumn{2}{c|}{\textcolor{red}{$\alpha = 0.4$}} & \multicolumn{2}{c|}{\textcolor{red}{$\alpha = 0.5$}} & \multicolumn{2}{c|}{\textcolor{red}{$\alpha = 0.7$}}\\
					%								\cline{2-7}
					%								&  & \textcolor{plum}{Error} & \textcolor{plum}{OC} & \textcolor{plum}{Error} & \textcolor{plum}{OC} & \textcolor{plum}{Error} & \textcolor{plum}{OC}  \\	
					%%%  FOR BLACK FONT
					\multirow{3}{*}{\large $M_s$} & \multicolumn{2}{c|}{$\alpha_1 = 0.5$, $\alpha_2 = 0.47$} & \multicolumn{2}{c|}{$\alpha_1 = 0.9$, $\alpha_2 = 0.88$, $\alpha_3 = 0.86$} \\
					& \multicolumn{2}{c|}{$\alpha_3 = 0.45$, $\alpha_4 = 0.43$} & \multicolumn{2}{c|}{$\alpha_4 = 0.84$, $\alpha_5 = 0.82$} \\
					\cline{2-5}
					& Error & OC & Error & OC  \\
					%%%	
					\hline
					$2^4$ & 2.26E-02 & 1.009491 & 2.24E-02 & 1.001092 \\
					\hline
					$2^5$ & 1.12E-02 & 1.002898 & 1.12E-02 & 1.000286 \\
					\hline
					$2^6$ & 5.60E-03 & 1.000786 & 5.60E-03 & 1.000072 \\
					\hline
					$2^7$ & 2.80E-03 & - & 2.80E-03 & - \\
					\hline
			\end{tabular}
			\end{center}
			\caption {\emph {Error and order of convergence in $H^1_0(\Omega)$ norm in space for Example \ref{7E3}.}}
			\label{7T10}
		\end{table}
	\end{tiny}
\end{center}
%%%

\section{Conclusions}\label{7conclusion}
In this work, we have used $L2$-$1_{\sigma}$ scheme on a graded mesh to discretize the multi-term time-fractional derivative, which gives optimal convergence rate in case of having solution with weak singularity near $t=0$. Under certain regularity assumptions on exact solution we have derived $\alpha_1$-robust \emph{a priori} error estimates in $L^2(\Omega)$ and $H^1(\Omega)$ norms. We also have verified our theoretical estimates by considering numerical experiments. The existence-uniqueness result for the weak solution of \eqref{7e1} and the derivation of regularity assumptions (in Section~\ref{7err_est}) on exact solution $u$ are under investigation.  \\

%%%%%%%%%%%%%%%%%%%%%%%%%%%%%%%%%%%
%%%%%%%%%%%%%%%%%%%%%%%%%%%%%%%%%%

%\section{Acknowledgement}
\noindent \textbf{Declarations:}\\
\textbf{Conflict of interest-} The author declares no competing interests.

%%%%%%%%%%%%%%%%%%%%%%%%%%%%%%%%%%%

%%%%%%%%%%%%%%%%%%%%%%%%%%%%%%%%%%
\end{document}